\documentclass{scrartcl}

\usepackage{a4wide}
\usepackage{amsmath}
\usepackage{amsthm}
\usepackage{txfonts}
\usepackage{graphicx}
\usepackage{pict2e}
\usepackage{color}

\theoremstyle{plain}
\newtheorem{theorem}{Theorem}
\newtheorem{lemma}{Lemma}
\theoremstyle{definition}
\newtheorem{definition}{Definition}
\theoremstyle{remark}
\newtheorem{example}{Example}
\newtheorem{remark}{Remark}

\newcommand \Stackrel[2]{\mathrel{\makebox[.8em][c]{$\stackrel{#1}{#2}$}}}
\newcommand \sumbox[1]{\makebox[3em][c]{$\displaystyle#1$}}
\newcommand \barOmega {\overline{\rule{0pt}{7pt}\smash\varOmega}}
\newcommand \tcup    {\operatorname*{\textstyle\bigcup}}

\newcommand \J       {\mathcal J}
\newcommand \hJ      {\hat{\mathcal J}}
\newcommand \hphi    {\hat\varphi}
\newcommand \hj      {{\hat\jmath}}
\newcommand \homega  {{\hat\omega}}
\newcommand \Vc      {{V_{\mathrm c}}}
\newcommand \vc      {v_{\mathrm c}}
\newcommand \uc      {u_{\mathrm c}}
\newcommand \vct     {v_{\mathrm c}}
\newcommand \ucm     {u_{\mathrm c}^m}
\newcommand \uco    {u_{\mathrm c}^\varOmega}
\newcommand \Vcm     {{V_{\mathrm c}^m}}
\newcommand \Vco    {{V_{\mathrm c}^\varOmega}}
\newcommand \Vf      {{V_{\mathrm f}}}
\newcommand \Div     {\operatorname{div}}
\newcommand \supp    {\operatorname{supp}}
\newcommand \diam    {\operatorname{diam}}
\newcommand \dist    {\operatorname{dist}}
\newcommand \de      {\,\mathrm d}
\newcommand \rhs     g
\newcommand \ceq     {\mathrel{\raisebox{.06ex}:{=}}}
\newcommand \qedsquare     {}
\newcommand \raiseqedsquare     {\\[-3.25em]}
\newcommand \subsetsim[1][]{\mathrel{\rule[-.45em]{0pt}{1.7em}\smash{\underset{\raisebox{5pt}[0pt][0pt]{$\scriptstyle\sim$}}{\stackrel{#1}\subset}}}}

\DeclareGraphicsExtensions {.jpg,.pdf}

\title{Multiscale Partition of Unity}

\author{Patrick Henning\footnote{ANMC, Section de Math\'{e}matiques, \'{E}cole polytechnique f\'{e}d\'{e}rale de
Lausanne -- \texttt{patrick.henning@epfl.ch}}
\\
Philipp Morgenstern\footnote{Institut f\"ur Numerische Simulation, Rheinische Friedrich-Wilhelms-Universit\"at Bonn -- \texttt{morgenstern@ins.uni-bonn.de}} 
\\
Daniel Peterseim\footnote{Institut f\"ur Numerische Simulation, Rheinische Friedrich-Wilhelms-Universit\"at Bonn -- \texttt{peterseim@ins.uni-bonn.de}}}

\begin{document}

\maketitle

\begin{abstract}
  \textbf{Abstract.} We introduce a new Partition of Unity Method for the numerical homogenization of elliptic partial differential equations with arbitrarily rough coefficients. We do not restrict to a particular ansatz space or the existence of a finite element mesh. The method modifies a given partition of unity such that optimal convergence is achieved independent of oscillation or discontinuities of the diffusion coefficient. The modification is based on an orthogonal decomposition of the solution space while preserving the partition of unity property. This precomputation involves the solution of independent problems on local subdomains of selectable size. We deduce quantitative error estimates for the method that account for the chosen amount of localization. Numerical experiments illustrate the high approximation properties even for `cheap' parameter choices. 
  
\end{abstract}

\section{Introduction}
\label{HMP::sec:introduction}
In this paper, we present a novel Multiscale Partition of Unity Method for reliable numerical homogenization in the meshfree context. 

The Partition of Unity Method (PUM) was introduced by Babu\v ska  and Melenk in \cite{BM96,Melenk:Babuska:1996}, with the motivation that known singularities of the solution of a given PDE can be embedded into the ansatz space. Examples of Partition of Unity Methods can be found in \cite{Duarte:Oden:1996,Griebel:Schweitzer:2000,Griebel:Schweitzer:2002,Holst:2003,Liszka:Duarte:Tworzydlo:1996,Oden:Duarte:Zienkiewicz:1998,Wang:Huang:Li:2008}.
Specific realizations of methods that fit into the general PUM framework but which are formulated 
in the context of finite element methods are the Extended Finite Element Method (XFEM, cf. \cite{Belytschko:Moes:Usui:Parimi:2001,Belytschko:Moes:Usui:Parimi:1999}), the Generalized Finite Element Method (GFEM, cf. \cite{Duarte:Babuska:Oden:2000,Duarte:Kim:2008,Duarte:Reno:Simone:2007,Kim:Duarte:Proenca:2012,Strouboulis:Babuska:Copps:2000,Strouboulis:Babuska:Copps:2001}) and the Stable GFEM presented in \cite{Gupta:Duarte:Babuska:Banerjee:2013}. More general surveys on XFEM and GFEM can be found in \cite{Babuska:Banerjee:Osborn:2003,Fries:Matthies:2004,Schweitzer:2012}.

In contrast to local singularities (usually due to the shape of the domain), multiscale problems consider the issue of very rough coefficients all over the domain. In order to obtain a reliable numerical approximation to the solution of the multiscale problem, it is typically necessary to 'resolve the coefficient', whereas a simple local averaging of the coefficient leads to wrong approximations. This means that the discrete solution space in which we seek an adequate Galerkin approximation must be able to fully capture the fine structures of the coefficient. Practically, this often leads to very large spaces and therefore to tremendous computational efforts. One approach to overcome this difficulty is to construct a special low dimensional space that incorporates the relevant fine scale features in its basis functions and that exhibits high approximation properties. A locally supported basis of this space can be computed in parallel by solving fine scale problems in small patches. This approach has been 
studied extensively for Finite Elements in \cite{HMP12,HP13,MP11,MP12}.

Other numerical multiscale methods can be found in \cite{Abdulle:E:Engquist:Vanden-Eijnden:2012,Babuska:Caloz:Osborn,E:Engquist:2003,Gloria:2011,Hou:Wu:1997,MR1660141,Malqvist:2011}. In the context of meshfree methods we refer to recent papers \cite{MR2801210, BOZ13} where elliptic problems with rough coefficients are treated by introducing special non-polynomial shape functions, i.e., local eigenfunctions in \cite{MR2801210} and rough polyharmonic splines in \cite{BOZ13}. 

This paper aims to generalize the mesh-based approach of \cite{HMP12,HP13,MP11,MP12} to general ansatz spaces without the requirement of underlying finite element meshes. 

Throughout the paper, our model problem consists of finding a stationary heat distribution in some heterogenous media.
Let $A\in L^\infty(\varOmega,\mathbb{R}^{d\times d}_{sym})$ be a symmetric coefficient with uniform spectral bounds $\beta\geq\alpha>0$ in some bounded Lipschitz domain $\varOmega\subset\mathbb R^d$ for $d=1,2,3$, i.e.,
\begin{alignat*}{2}
 0&<\alpha&&\ceq\operatorname*{ess\,inf}\limits_{x\in\varOmega}\inf_{v\in\mathbb R^d\setminus\{0\}}\frac{\bigl(A(x)v,v\bigr)}{(v,v)},\\
 \infty&>\beta&&\ceq\operatorname*{ess\,sup}_{x\in\varOmega}\sup_{v\in\mathbb R^d\setminus\{0\}}\frac{\bigl(A(x)v,v\bigr)}{(v,v)}.
\end{alignat*}

This coefficient $A$ may be strongly heterogenous and arbitrarily rough.
We consider the prototypical second-order linear elliptic PDE
\begin{equation}\tag{1a}
    -\Div A\nabla u = \rhs
\end{equation}
with homogeneous Neumann boundary condition
\begin{equation}\tag{1b}
A \nabla u \cdot \nu = 0 \quad \text{on }\partial \varOmega,
\end{equation}
\stepcounter{equation}%
given the exterior normal vector $\nu$ on $\partial\varOmega$ and compatible right-hand side $\rhs \in L^2(\varOmega)$ such that $$\int_\varOmega \rhs\de x=0.$$

We are looking for the unique (up to a constant) weak solution of problem (1a--b). This is, for $V\ceq H^1(\varOmega)$, find $u \in V/\mathbb R=\{v\in V\mid \int_\varOmega v\de x=0\}$ with
\begin{equation}\label{HMP::eq:weakPMP}
a(u,\phi)\ceq\int_\varOmega A \nabla u \cdot \nabla \phi \de x= \int_\varOmega \rhs \phi \de x \quad \text{for all } \phi \in V /\mathbb R.
\end{equation}

\section{Abstract Multiscale Partition of Unity}
In this section, we propose a Multiscale Partition of Unity Method without restriction to a particular ansatz space
or even the existence of a mesh.
This method is built upon two abstract (and possibly equal) partitions of unity that will be introduced in Section~\ref{HMP::sec:pums}. Another crucial tool for the design of the method and its error analysis is a quasi-interpolation operator presented in Section~\ref{HMP::sec:interpolator}. In the third and last subsection, we finally define the novel multiscale partition of unity method based on a localized orthogonal decomposition of $V$.

\subsection{Two Partitions of Unity}\label{HMP::sec:pums}
The subsequent derivation of the multiscale method is based upon two standard partitions of unity.
One partition is regular and spans a coarse space $\Vc$. The other partition may be discontinuous and is solely used for the localization of the corrector problems in Section~\ref{HMP::sec:mspum}.

\begin{definition}[Partitions of Unity]\label{HMP::def:PoU}\ \\[-1.5em]
\begin{itemize}
\item[(PU 1)] Let $\J$ denote a finite index set and $\{ \varphi_j\mid j\in \J \}$ a linearly independent Lipschitz partition of unity on $\varOmega$, i.e.
\begin{align*}
\sum_{j\in\J} \varphi_j = 1\quad\text{with}\quad&\forall\ j\in\J:\enspace0\le\varphi_j \in W^{1,\infty}(\varOmega),\\
\text{s.t. for any }&\lambda\in\mathbb R^\J,\enspace\sum_{j\in\J}\lambda_j\varphi_j = 0\quad\Leftrightarrow\quad\forall\ j\in\J:\enspace\lambda_j=0.
\end{align*}
We define $\omega_j\ceq\supp(\varphi_j)$ and $H_j\ceq \diam(\omega_j)$ for all $j\in\J$ and \text{$H\ceq\max_{j\in\J}H_j$}.
The partition of unity functions span a finite dimensional coarse space $\Vc\ceq\operatorname{span}\{ \varphi_j\mid j\in\J\}$. 
\item[(PU 2)] Let $\hJ$ denote a finite index set and $\{ \hphi_\hj\mid\hj\in\hJ \}\subseteq L^\infty(\varOmega)$ a bounded and positive partition of unity on $\varOmega$, i.e.
\begin{displaymath}
\sum_{\smash{\hj\in\hJ}} \hphi_\hj = 1 \quad \text{on} \enspace \varOmega \quad \text{and} \quad \hphi_\hj\ge0.
\end{displaymath}
We define $\homega_\hj\ceq\supp(\hphi_\hj)$ and $\hat H_\hj\ceq\diam(\homega_\hj)$  for all $\hj\in\hJ$. The maximum over all $\hat H_\hj$ is denoted by $\hat H$. 
\end{itemize}
\end{definition}

\begin{example}\label{HMP::ex:1}
The abstract definitions of (PU 1) and (PU 2) include the following special cases.
\begin{itemize}
\item[a)] (PU 2) equals (PU 1).
\item[b)] Given some regular simplicial mesh $\mathcal T$ with vertices $\mathcal N=\J$, the partition (PU 1) is the continuous piecewise affine nodal basis functions $\varphi_z$, associated with vertices $z\in\mathcal N$. Recall that $\varphi_z$ is defined by its values $\varphi_z(y)=\left\{\begin{smallmatrix}1&\text{if }y=z\\0&\text{else}\end{smallmatrix}\right.$ for vertices $y\in\mathcal N$.
(PU 2) may be chosen as the characteristic (or `indicator') functions of the triangles, i.e. $$\hJ=\mathcal T\quad\text{and}\quad\hphi_T=\chi_T\text{ for all }T\in\mathcal T.$$
\end{itemize}
\end{example}
\begin{definition}[extension patch]
\label{definition:extensionpatch}For any patch $\omega_j$ in (PU 1) and $k \in \mathbb{N}$, we define the \emph{$k$-th order extension patch} $\omega_j^k$ by
\begin{displaymath}
\omega_j^k \ceq \tcup_{x\in\omega_j}\overline{B_{k\cdot H}(x)} = \left\{ x\in\barOmega\mid \dist(x,\omega_j)\leq k\cdot H\right\}.
\end{displaymath}
where $B_{k\cdot H}(x)$ denotes the ball with radius $k\cdot H$ around $x$ and where ``$\dist$'' denotes the set distance $$\dist(x,B)\ceq\inf_{b\in B}\left\|x-b\right\|.$$
For (PU 2), the extension patches $\homega_\hj^k$, $k\in\mathbb N$ are defined analogously. 
\end{definition}
The subsequent definition serves only for the proofs. It has no practical relevance for the proposed method.
\begin{definition}[quasi-inclusion]
Given two sets $B,C\subseteq\barOmega$, the set $B$ is \emph{$n$-quasi-included} in $C$ (shorthand notation: $B\subsetsim[n] C$) if $$\forall\ j_1,\dots, j_m\in\J,\enspace k_1,\dots,k_m\in\mathbb N:\enspace C\subseteq \tcup_{i=1}^m\omega_{j_i}^{k_i} \enspace\Rightarrow \enspace B\subseteq\tcup_{i=1}^m\omega_{j_i}^{k_i+n}.$$ Note that the shorthand notation allows for quantified transitivity $$B\subsetsim[n_1]C\subsetsim[n_2]D\quad\Rightarrow\quad B\subsetsim[n_1+n_2]D.$$
\end{definition}

\subsection{Abstract Quasi-Interpolation}\label{HMP::sec:interpolator}
\begin{definition}[quasi-interpolation operator]\label{HMP::def:quasi-interpolation operator}
Throughout this paper, let $I:V \to \Vc$ denote an abstract quasi-interpolation operator which fulfills the following properties.
\begin{itemize}
\item[\textbf{(I1)}] $I$ is linear and continuous.
\item[\textbf{(I2)}] $I|_\Vc: \Vc\to \Vc$ is an isomorphism with $H^1$-stable inverse.
\item[\textbf{(I3)}] There exists a constant $C_1$ only depending on $\varOmega$ and the shape of the patches $\omega_j$ such that for all $u\in H^1(\varOmega)$ and all $j\in\J$
\begin{displaymath}
\| u - I(u) \|_{L^2(\omega_j)} \le C_1 H_j \| \nabla u \|_{L^2(\omega_j^1)},
\end{displaymath}
and a constant $C_2$ that further depends on $\max_{j\in\J} \bigl( H_j\left\|\varphi_j\right\|_{W^{1,\infty}(\varOmega)} \bigr)$ such that
\begin{displaymath}
\| \nabla I(u) \|_{L^2(\omega_j)} \le C_2 \| \nabla u \|_{L^2(\omega_j^1)}.
\end{displaymath}
\item[\textbf{(I4)}] There exists a constant $C_3$ with same dependencies as $C_2$ and some $\kappa\in\mathbb N$ depending on the overlapping of the supports $\{\omega_j\}_{j\in\J}$ such that for all $\vc \in \Vc$ there exists $v \in V$ such that
\begin{displaymath}
I(v)= \vc, \quad \|\nabla v \|_{L^2(\varOmega)} \le C_3 \| \nabla \vc \|_{L^2(\varOmega)}, \enspace\text{and} \enspace\supp(v) \subsetsim[\kappa] \supp(\vc),
\end{displaymath}
with the quasi-inclusion $\subsetsim[\kappa]$ defined above.
\end{itemize}
\end{definition}
A particular quasi-interpolation operator $I$ is given in the subsequent definition.

\begin{example}[Clement-type quasi-interpolation \cite{Car99}]\label{HMP::ex:Clement}
Define a weighted Cl\'ement-type quasi-interpolation operator
$$ I : V \to \Vc ,\quad v\mapsto I(v)\ceq \sum_{j\in\J} v_j \varphi_j \quad \text{with }v_j \ceq \frac{(v,\varphi_j)_{L^2(\varOmega)}}{(1,\varphi_j)_{L^2(\varOmega)}}.$$
\end{example}
This operator obviously satisfies \textbf{(I1)} and \textbf{(I2)}.
The properties \textbf{(I3)} have been shown in \cite{Car99} in the abstract setting of (PU 1). We verify that \textbf{(I4)} is satisfied for a particular choice of basis functions. The following result is similar to \cite[Lemma 2.1]{MP11}.\bigskip

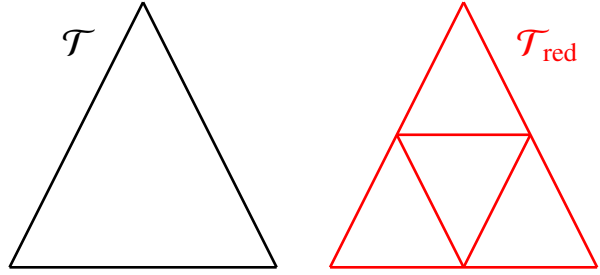
\begin{figure}[t]
\centering
\begin{picture}(220,150)
\linethickness{1.0pt}
\put(0,0){\line(1,0){100}}
\put(0,0){\line(1,2){50}}
\put(100,0){\line(-1,2){50}}
\put(20,80){{\Large$\mathcal T$}}
\linethickness{1.0pt}{\color{red}
\put(120,0){\line(1,0){100}}
\put(120,0){\line(1,2){50}}
\put(220,0){\line(-1,2){50}}
\put(145,50){\line(1,0){50}}
\put(145,50){\line(1,-2){25}}
\put(170,00){\line(1,2){25}}
\put(190,80){{\Large$\mathcal T_{\mathrm{red}}$}}}
\end{picture}
\caption{Sketch of a so-called {\it red refinement} of a single triangle in $2d$. In general, the red refinement is based on the bisection of all edges and yields at least $d+1$ simplices of same shape and half diameter.}
\label{red-refinement-figure}
\end{figure}

\begin{lemma} For a given regular triangulation $\mathcal T$ with vertices  $\mathcal N=\J$ and nodal basis functions $\left\{\varphi_z\right\}_{z\in\mathcal N}$ as in Example~\ref{HMP::ex:1}b), the quasi-interpolation operator from Example~\ref{HMP::ex:Clement} satisfies \textbf{(I4)} with $\kappa=1$.
\end{lemma}

\begin{proof}
For any basis function $\varphi_z$, we want to find $b_z \in H^1(\varOmega)$ with $$I(b_z)=\varphi_z,\quad|\nabla b_z| \le C |\nabla \varphi_z| \text{ a.e. in $\varOmega$ and}\quad\supp(b_z)\subseteq\supp(\varphi_z).$$

Consider the red refinement $\mathcal T_{\mathrm{red}}$ of $\mathcal T$ (cf. Figure \ref{red-refinement-figure}) with nodal basis functions $\varphi_z^{\mathrm r}$. If $\operatorname{nb}(z)$ denotes the set of all neighboring nodes of $z$ in $\mathcal T_{\mathrm{red}}$ it can be verified that $$b_z=(2^{d+1}-1)\varphi_z^{\mathrm r}-\tfrac1{2}\hspace{-3pt}\sum_{y\in\operatorname{nb}(z)}\varphi_y^{\mathrm r}$$ satisfies the desired conditions.

To conclude the proof, set $$w\ceq \sum_{z\in\mathcal N}\left(\vc(z)-I(\vc)(z)\right)b_z$$
and observe that $v\ceq \vc+w$ satisfies $\vc=I(v)$, with $\supp(v)\subsetsim[1]\supp(\vc)$ and $$\|\nabla v\|\leq (1+C+C_2C)\left\|\nabla \vc\right\|.$$
\raiseqedsquare\end{proof}

\subsection{Definition of the method}\label{HMP::sec:mspum}
The goal is the construction of a space $\Vcm$ that is of the same dimension as the discrete coarse space $\Vc=\operatorname{span}\{ \varphi_j\mid j\in\J\}$ (cf.\ Definition~\ref{HMP::def:PoU}) but which exhibits high $H^1$-approximations that are inherited from the $L^2$-approximation properties of $\Vc$. Furthermore, we wish to explicitly construct a partition of unity basis for $\Vcm$.

Under the conditions \textbf{(I1)} and \textbf{(I2)} on the abstract quasi-interpolation operator, the space $V$ can be written as the direct sum
\begin{equation}\label{HMP::def-W_h}
V = \Vc\oplus \Vf, \quad \text{with} \enspace \Vf \ceq \{ v \in V \mid I(v) = 0 \}.
\end{equation}
The subspace $\Vf$ contains the fine scale features in $V$ that cannot be captured by the coarse space $\Vc$.

\begin{definition}[corrector]\label{HMP::def:corrector}
For $\hj\in\hJ$ and $m\in\mathbb N$, define the \emph{local corrector} $Q_\hj^m: \Vc\to \Vf(\homega_\hj^m)$ as the mapping of a given $\vc \in \Vc$ onto the solution $Q_\hj^m(\vc) \in \Vf(\homega_\hj^m)\ceq\bigl\{v \in \Vf\mid v=0 \enspace \text{in } \varOmega \setminus \homega_\hj^m\bigr\}$ of 
\begin{equation}
\label{HMP::local-corrector-problem}\int_{\homega_\hj^m} A \nabla Q_\hj^m(\vc)\cdot \nabla w \de x= - \int_{\homega_\hj} {\hphi}_\hj A \nabla \vc \cdot \nabla w \de x \qquad \text{for all } w \in \Vf(\homega_\hj^m).
\end{equation}
The global corrector is given by
\begin{displaymath}
Q^m(\vc)\ceq\sum_{\smash{\hj\in\hJ}} Q_\hj^m(\vc).
\end{displaymath}
For sufficiently large $m$ such that $\homega_\hj^m=\barOmega$ for all $\hj\in\hJ$, we call $Q^\varOmega\ceq Q^m$ 
the \emph{ideal corrector}.
\end{definition}
The parameter $m$ in Definition~\ref{HMP::def:corrector} reflects the locality of the method. The computational cost grows polynomially with $m$, while the error decays exponentially towards the error of the ideal (not localized) method.

Observe that the corrector problem \eqref{HMP::local-corrector-problem} always yields a unique solution. Existence is clear by the Lax-Milgram theorem, because the zero function is the only constant function in $\Vf$. For any $m\in\mathbb N$, the operator $Q^m$ is linear and we denote the corrected discrete space 
\begin{equation}\label{HMP::eq:corrected space}
\Vcm\ceq\{ \vc + Q^m(\vc)\mid \vc \in \Vc\},\quad
\Vco\ceq\{ \vc + Q^\varOmega(\vc)\mid \vc \in \Vc\}.
\end{equation}
Note that $\Vcm$ (and also $\Vco$) satisfies
\begin{displaymath}
 V=\Vcm\oplus\Vf
\end{displaymath}
and that $\{ \varphi_j + Q^m(\varphi_j)\mid j\in\J\}$ is a basis of $\Vcm$.
Moreover, the ideal method comes with $a$-orthogonality of $\Vc$ onto $\Vf$, i.e.
\begin{equation}
\label{HMP::eq:a-orthogonality}
a(\Vco,\Vf)=0.
\end{equation}

\begin{remark} The partition of unity property is preserved under correction. To prove this, it suffices to show $\sum_{j\in\J} Q^m(\varphi_j) = 0$. We compute
\begin{displaymath}
\sum_{j\in\J} Q^m(\varphi_j) = \sum_{j\in\J} \sum_{\smash{\hj\in\hJ}} Q_\hj^m(\varphi_j) = \sum_{\smash{\hj\in\hJ}} Q_\hj^m \Bigl( \sum_{j\in\J} \varphi_j \Bigr) = \sum_{\smash{\hj\in\hJ}} Q_\hj^m( 1 ) = 0.
\end{displaymath}
Hence $\{ \varphi_j + Q^m(\varphi_j)\mid j\in\J\}$ is a partition of unity. This also holds for the ideal corrector $Q^\varOmega$.
\end{remark}

The Galerkin discretization of \eqref{HMP::eq:weakPMP} with respect to the corrected space $\Vcm$, $m\in\mathbb N$ reads as follows.
\begin{definition}[Multiscale Partition of Unity Method]
Find $\ucm\in\Vcm /\mathbb R$ such that
\begin{align}\label{HMP::eq:discretePMP} 
\int_\varOmega A \nabla \ucm \cdot \nabla \vct \de x
&= \int_\varOmega \rhs\vct \de x\quad\text{for all }\vct \in \Vcm /\mathbb R.
\intertext{The \emph{ideal problem} seeks $\uco\in\Vco /\mathbb R$ such that}
\label{HMP::eq:idealDiscretePMP} \int_\varOmega A \nabla \uco \cdot \nabla \vct \de x
&= \int_\varOmega \rhs\vct \de x\quad\text{for all }\vct \in \Vco /\mathbb R.
\end{align}
\end{definition}

\section{A priori error analysis}\label{HMP::sec:estimates}
In this section, we prove error estimates for the discrete solution of \eqref{HMP::eq:discretePMP}. In the first subsection, we consider the ideal case with ansatz space $\Vco$ (cf. \eqref{HMP::eq:corrected space}). The second subsection yields an error estimate for the localized problem. We will use the notation ``$a\lesssim b$'' to state the existence of $C>0$ such that $a\leq Cb$. The hidden constant $C$ may depend on the Poincar\'e constant $C_\mathrm{Poinc}(\varOmega)$, on the ratio ${\hat H}/H$, on the constants $C_1,C_2,C_3$ and $\kappa$ from \textbf{(I1)}--\textbf{(I4)} in Definition~\ref{HMP::def:quasi-interpolation operator}, and on the operator norms of $I$ and $\big(I|_\Vc\big)^{-1}$ that result from \textbf{(I1)} and \textbf{(I2)}. The hidden constant does \emph{not} depend on the data $A$ and $\rhs$, the spectral bounds $\alpha$ and $\beta$ (in particular the contrast $\frac\beta\alpha$) or the patch sizes $H$ and $\hat H$.

\subsection{Error estimate for global basis functions}
We consider the ideal (but expensive) case of no localization (i.e. $\homega_\hj^m=\barOmega$) and observe that the proposed method inherits the optimal approximation properties. This estimate is also important in the analysis of the localized method in Section~\ref{HMP::sec:local-estimates}.

\begin{theorem}[A priori error estimate for the ideal case]
\label{HMP::thm:global-basis-error}
Let $u$ be the solution of \eqref{HMP::eq:weakPMP}. Then the discrete solution $\uco$ of \eqref{HMP::eq:idealDiscretePMP} satisfies
\begin{displaymath}
\alpha^{1/2}\bigl\|  \nabla(\uco - u) \bigr\|_{L^2(\varOmega)}\leq \bigl\| A^{1/2} \nabla(\uco - u) \bigr\|_{L^2(\varOmega)} \lesssim \alpha^{-1/2}H \| \rhs \|_{L^2(\varOmega)}.
\end{displaymath}
\end{theorem}

\begin{proof}
Observe that we can replace the test function space $\Vco/\mathbb R$ by $\Vco$, since we subsequently only consider gradients. Galerkin orthogonality, i.e.
\begin{equation}
\label{HMP::eq:Galerkin-orthogonality}
a\left( u -  \uco ,\vct \right) = 0\quad\text{for all }\vct \in \Vco
\end{equation}
and \eqref{HMP::eq:a-orthogonality} imply that $e\ceq u-\uco\in\Vf$ and therefore $I(e)=0$. We get
\begin{align*}
\|A^{1/2}\nabla e \|_{L^2(\varOmega)}^2 &= a(e,e)\\
&\Stackrel{\eqref{HMP::eq:Galerkin-orthogonality}}= a(e,u) = a(u,e)\\
&= \int_\varOmega \rhs ( e-I( e ) ) \de x \\
&\Stackrel{\textbf{(I3)}}\lesssim H \| \rhs \|_{L^2(\varOmega)} \|  \nabla e \|_{L^2(\varOmega)}\\
&\leq \alpha^{-1/2}H \| \rhs \|_{L^2(\varOmega)} \| A^{1/2} \nabla e \|_{L^2(\varOmega)}.
\end{align*}
\raiseqedsquare
\end{proof}
\subsection{Error estimate for local basis functions}
\label{HMP::sec:local-estimates}
In this final subsection, we give error estimates for the localized method. The main result is presented below.
\begin{theorem}[A priori error estimates for the localized method]\label{HMP::t:a-priori-local}
Assume that $u\in V$ solves \eqref{HMP::eq:weakPMP}, then the discrete solution $\ucm\in\Vcm$ of \eqref{HMP::eq:discretePMP} satisfies
\begin{align*}
 \|\nabla u -\nabla \ucm\|_{L^2(\varOmega)}&\lesssim \alpha^{-1}\bigl(H + \tfrac\beta\alpha m^{d/2} \tilde\theta^m \bigr) \|\rhs\|_{L^2(\varOmega)},\\
 \| u - \ucm \|_{L^2(\varOmega)}&\lesssim \alpha^{-2}\bigl(H + \tfrac\beta\alpha m^{d/2} \tilde\theta^m \bigr)^2 \|\rhs\|_{L^2(\varOmega)},
\end{align*}
with some generic constant $0<\tilde\theta = \theta^{\lceil \hat H/H\rceil}<1$ and $\theta$ depending on the contrast $\frac\beta\alpha$ (cf. Lemma \ref{HMP::lma:decay} and \ref{HMP::lma:localization-error} below).
\end{theorem}

\begin{proof}
Let $\uco$ be the solution of the ideal problem with correction operator $Q^\varOmega$, and $\uc\in\Vc/\mathbb R$ such that $\uco=\uc+Q^\varOmega\uc$. As a consequence of \textbf{(I3)}, all functions $v\in\Vf$ satisfy $\int_\varOmega v\de x=0$. With $Q^m\vc\in\Vf$, we get 
\begin{align*}
\|\nabla u-\nabla \ucm\|_{L^2(\varOmega)}
&\lesssim \min_{\vc^m\in\Vcm/\mathbb R} \left\|\nabla u - \nabla\vc^m \right\|_{L^2(\varOmega)}\\
&\le \left\|\nabla u - \nabla(\uc + Q^m\uc) \right\|_{L^2(\varOmega)}\\
&\le \left\|\nabla u - \nabla\uco \right\|_{L^2(\varOmega)} + \left\| \nabla Q^\varOmega\uc - \nabla Q^m\uc \right\|_{L^2(\varOmega)}.
\end{align*}
Lemma~\ref{HMP::lma:localization-error} will quantify the localization error
$$\left\| \nabla Q^\varOmega\uc - \nabla Q^m\uc \right\|_{L^2(\varOmega)}\lesssim\tfrac\beta\alpha m^{d/2} \tilde\theta^m \Bigl( \sum_{\smash{\hj\in\hJ}} \|\nabla Q^\varOmega_\hj\uc \|_{L^2(\varOmega)}^2  \Bigr)^{1/2}.$$
This, Theorem~\ref{HMP::thm:global-basis-error} and the estimates
\begin{align*}
\sum_{\smash{\hj\in\hJ}} \|\nabla Q^\varOmega_\hj\uc \|_{L^2(\varOmega)}^2 &\stackrel{\eqref{HMP::local-corrector-problem}}\lesssim \sum_{\smash{\hj\in\hJ}}  \|{\hphi}_\hj \nabla \uc \|_{L^2(\varOmega)}^2 \le  \|\nabla \uc \|_{L^2(\varOmega)}^2 \\
&=\bigl\|\nabla\big(I|_\Vc\big)^{-1}I(\uco)\bigr\|_{L^2(\Omega)}\stackrel{\textbf{(I1)},\textbf{(I2)}}\lesssim \|\nabla \uco \|_{L^2(\varOmega)}^2
\le \alpha^{-1}C_\mathrm{Poinc}(\varOmega)\,\|\rhs\|_{L^2(\varOmega)}^2
\end{align*}
yield the $H^1$-error estimate. The $L^2$-error estimate is obtained by a standard Aubin-Nitsche argument.
\qedsquare\end{proof}
To prove Lemma~\ref{HMP::lma:localization-error}, several tools are needed in addition to the preceding results. They will be discussed below.
\begin{lemma}[quasi-inclusion of intersecting patches]\label{HMP::lma:set-comp}
Let $i,j\in\J$ and $\ell,k,m \in \mathbb{N}$ with $k\ge \ell \ge 2$. Then
\begin{displaymath}
\text{if}\quad\omega_i^m \cap \left( \omega_j^k \setminus \omega_j^\ell \right)\neq \emptyset \quad \text{then}\quad \omega_i \subseteq \omega_j^{k+m+1} \setminus \omega_j^{\ell-m-1}.
\end{displaymath}
\end{lemma}

\begin{proof}
Consider $x\in\omega_i^m \cap \left( \omega_j^k \setminus \omega_j^\ell \right)$ and observe
$$ \omega_i\enspace \subseteq \enspace\overline{ B_{(m+1)H}(x)} \enspace\subseteq \enspace\omega_j^{k+m+1} \setminus \omega_j^{\ell-m-1}.$$
\raiseqedsquare\end{proof}

\begin{definition}[cut-off functions]\label{HMP::def:cutoff}
For all $j\in\J$ and $\ell,k\in\mathbb{N}$ with $k > \ell$, we define the \emph{cut-off function}
\begin{displaymath}
\eta_j^{k,\ell}(x) = \frac{\dist(x,\omega_j^{k-\ell})}{\dist(x,\omega_j^{k-\ell})+\dist(x,\varOmega\setminus\omega_j^k)}.
\end{displaymath}
For $\varOmega\setminus\omega_j^k=\emptyset$, we set $\eta_j^{k,\ell}\equiv0$.
Note that $\eta_j^{k,\ell}=0$ in $\omega_j^{k-\ell}$ and $\eta_j^{k,\ell}=1$ in $\varOmega\setminus\omega_j^k$. Moreover, $\eta_j^{k,\ell}$ is bounded between 0 and 1 and Lipschitz continuous with
\begin{equation}\label{HMP::e:cutoffH}
 \bigl\| \nabla\eta_j^{k,\ell} \bigr\|_{L^{\infty}(\varOmega)}\leq\frac 1{\ell\,H}.
\end{equation}
See \cite[Theorem 8.5]{Alt2006} for existence and boundedness of the weak derivative of Lipschitz-continuous functions.
\end{definition}
\begin{remark} The Lipschitz bound is shown as follows. For $x\in\mathbb R^d$ we have the triangle inequality
$$\dist(x,\omega_j^{k-\ell}) + \dist(x,\varOmega\setminus\omega_j^k)\enspace\geq\enspace\dist(\omega_j^{k-\ell},\varOmega\setminus\omega_j^k)\enspace=\enspace\ell\,H .$$
Moreover, any nonemtpy set $B$ in a metric space satisfies Lipschitz continuity of the distance function $\dist(\,\cdot\,,B)$ in the sense
$$\left|\dist(x,B)-\dist(y,B)\right|\enspace\leq\enspace \dist(x,y)\quad\text{for }x,y\in\mathbb R^d.$$
Altogether,
\begin{align*}
 \frac{\bigl|\eta_j^{k,\ell}(x)-\eta_j^{k,\ell}(y)\bigr|}{\dist(x,y)}
 &\leq\frac1{\dist(x,y)}\cdot\frac{\bigl|\dist(x,\omega_j^{k-\ell})-\dist(y,\omega_j^{k-\ell})\bigr|}{\ell\,H}\\
 &\leq \frac1{\ell\,H}.
\end{align*}
\end{remark}
A technical issue in our error analysis is that $\Vf$ is not invariant under multiplication by such cut-off functions. However, the product $\eta_j^{k,\ell}w$ for $w\in\Vf$ is close to $\Vf$ in the following sense.
\begin{lemma}[quasi-invariance of $\Vf$ under multiplication by cut-off functions]\label{HMP::lma:cutoff}
Recall $\kappa$ from \textbf{(I4)}. For any given $w\in \Vf$ and cutoff function $\eta_j^{k,\ell}$ with $k>\ell>0$, there exists $\tilde w\in \Vf(\varOmega \setminus \omega_j^{k-\ell-\kappa-2}) \subseteq \Vf$ such that
\begin{displaymath}\label{HMP::lemma-a-1-eq}
\|\nabla(\eta_j^{k,\ell}w-\tilde w)\|_{L^2(\varOmega)}\lesssim \ell^{-1} \|\nabla w\|_{L^2(\omega_j^{k+2} \setminus \omega_j^{k-\ell-2})}.
\end{displaymath}
\end{lemma}

\begin{proof}
We fix the $j \in\J$ and $k\in\mathbb{N}$ and denote $\eta_\ell\ceq\eta_j^{k,\ell}$ and $\textstyle c_i^{\ell} \ceq \frac1{|\omega_i^1|}\int_{\omega_i^1} \eta_\ell \de x$ for $i\in \J$. The property \textbf{(I4)}, applied to $I(\eta_\ell w)\in\Vf$, yields $v\in V$ with
\begin{gather}\label{HMP::e:lackproj}
I(v) =I (\eta_\ell w),\;\|\nabla v\|_{L^2(\varOmega)}\lesssim \|\nabla I \eta_\ell w\|_{L^2(\varOmega)},\\
\text{and}\notag\quad
\supp(v)\subsetsim[\kappa] \supp \bigl(I(\eta_\ell w)\bigr)\subsetsim[1]\supp(\eta_\ell w) \subseteq \varOmega \setminus \omega_j^{k-\ell},\\
\text{which yields}\quad \supp(v)\subsetsim[\kappa+1]\varOmega\setminus\omega_j^{k-\ell}\Rightarrow\supp(v)\subseteq\varOmega\setminus\omega_j^{k-\ell-\kappa-2}.
\end{gather}
Note that $\supp \bigl(I(\eta_\ell w)\bigr)\subsetsim[1]\supp(\eta_\ell w)$ is a consequence of \textbf{(I3)},
and that \eqref{HMP::e:lackproj} implies $I(v- \eta_\ell w)=0$.

We define $\tilde w\ceq\eta_\ell w-v\in \Vf(\varOmega \setminus \omega_j^{k-\ell-\kappa-2})$.
Using $I(w)=0$, we obtain for any $i\in \J$
\begin{equation}
\label{HMP::loc-lagrange-stability}\left\|\nabla I (\eta_\ell w)\right\|_{L^2(\omega_i)} \Stackrel{\textbf{(I1)}}= \left\|\nabla I ((\eta_\ell -c_i^{\ell})w)\right\|_{L^2(\omega_i)} \Stackrel{\textbf{(I3)}}\lesssim 
\left\|\nabla ((\eta_\ell -c_i^{\ell})w)\right\|_{L^2(\omega_i^1)}.
\end{equation}
This gives us
\begin{align*}
\|\nabla I (\eta_\ell w)\|_{L^2(\varOmega)}^2 &\leq {\sum_{i\in\J}}\| \nabla I (\eta_\ell w)\|_{L^2(\omega_i)}^2\\
&\Stackrel{\eqref{HMP::loc-lagrange-stability}}\lesssim\,
\sum_{i\in\J} \bigl\|\nabla\bigl((\eta_\ell- c_i^{\ell})w\bigr)\bigr\|_{L^2(\omega_i^1)}^2 \\
&=\sumbox{\sum_{\substack{i\in\J: \\ \omega_i^1\cap (\omega_j^k \setminus \omega_j^{k-\ell}) \neq \emptyset}}}
\bigl\|\nabla\bigl(\bigl(\eta_\ell- c_i^{\ell} \bigr)w\bigr)\bigr\|_{L^2(\omega_i^1)}^2\\
&\Stackrel{\eqref{HMP::lma:set-comp}}\leq \sumbox{\sum_{\substack{i\in\J: \\ \omega_i \subseteq \omega_j^{k+2} \setminus \omega_j^{k-\ell-2}}}}\bigl\|\nabla\bigl(\bigl(\eta_\ell- c_i^{\ell} \bigr)w\bigr)\bigr\|_{L^2(\omega_i^1)}^2\\
&\lesssim\sumbox{\sum_{\substack{i\in\J: \\ \omega_i \subseteq \omega_j^{k+2} \setminus \omega_j^{k-\ell-2}}}}
\bigl\| (\nabla \eta_\ell)( w - I w) \bigr\|_{L^2(\omega_i^1)}^2 + \bigl\| \bigl(\eta_\ell- c_i^{\ell} \bigr)\nabla w \bigr\|_{L^2(\omega_i^1)}^2.
\end{align*}
Since $\nabla\eta_\ell\neq0$ only in $\omega_j^k \setminus \omega_j^{k-\ell}$ and $\bigl.(\eta_\ell-c_i^\ell)\bigr|_{\omega_i^1}\neq0$ only if $\omega_i^1$ intersects with $\omega_j^k \setminus \omega_j^{k-\ell}$, we have
\begin{align}
\notag&\lesssim \sumbox{\sum_{\substack{i\in\J: \\ \omega_i \subseteq\omega_j^{k+1} \setminus \omega_j^{k-\ell-1}}}} \bigl\| (\nabla \eta_\ell)( w - I w) \bigr\|_{L^2(\omega_i)}^2 +\sumbox{\sum_{\substack{i\in\J: \\ \omega_i \subseteq \omega_j^{k+1} \setminus \omega_j^{k-\ell-1}}}} \bigl\| \bigl(\eta_\ell- c_i^{\ell} \bigr)\nabla w \bigr\|_{L^2(\omega_i^1)}^2 \\
\notag&\lesssim H^2 \|\nabla \eta_\ell\|^2_{L^\infty(\varOmega)} \| \nabla w  \|_{L^2(
\omega_j^{k+1} \setminus \omega_j^{k-\ell-1}
)}^2 
\\\notag&\hspace{4cm} + \sumbox{\sum_{\substack{i\in\J: \\ \omega_i \subseteq \omega_j^{k+1} \setminus \omega_j^{k-\ell-1}}}}
\bigl\| \bigl(\eta_\ell- c_i^{\ell} \bigr)\nabla w \bigr\|_{L^2(\omega_i^1)}^2\\
\label{HMP::e:ferlg}&\Stackrel{\eqref{HMP::e:cutoffH}}\leq\ell^{-2} \bigl\| \nabla w  \bigr\|_{L^2(
\omega_j^{k+2} \setminus \omega_j^{k-\ell-2}
)}^2,
\end{align}
where we used the Lipschitz bound $\| \eta_\ell- c_i^{\ell} \|_{L^{\infty}(\omega_i^1)} \lesssim H \left\| \nabla \eta_\ell \right\|_{L^{\infty}(\omega_i^1)}$.
The combination of \eqref{HMP::e:lackproj} and \eqref{HMP::e:ferlg} readily yields the assertion,
\begin{align*}
\left\|\nabla(\eta_\ell w-\tilde w)\right\|_{L^2(\varOmega)}^2 &= \left\|\nabla v\right\|_{L^2(\varOmega)}^2 \stackrel{\eqref{HMP::e:lackproj}}\le \left\|\nabla I (\eta_\ell w)\right\|_{L^2(\varOmega)}^2\\
&\Stackrel{\eqref{HMP::e:ferlg}}\lesssim  \ell^{-2}\left\|\nabla w\right\|^2_{L^2(
\omega_j^{k+2} \setminus \omega_j^{k-\ell-2}
)}.
\end{align*}
\raiseqedsquare\end{proof}

A key result is the following.
\begin{lemma}[Exponential decay in the fine scale space]\label{HMP::lma:decay}
Consider some fixed $j\in\J$ and let $F\in (\Vf)'$ satisfy $F(w)=0$ for all $w \in \Vf(\varOmega \setminus \omega_j^\varrho)$ with $\varrho\ceq\bigl\lceil\frac{\hat H}H\bigr\rceil$. Let $p \in \Vf$ be the solution of
\begin{equation}
\label{HMP::generalized-corrector-problem}a(p, w) =F(w) \qquad \text{for all } w \in \Vf.
\end{equation}
Then there exists $0<\theta<1$ depending on the contrast $\frac\beta\alpha$ such that for all positive $k\in\mathbb{N}$ it holds
\begin{displaymath}
\|\nabla p \|_{L^2(\varOmega\setminus \omega_j^k )}\lesssim \theta^k\left\|\nabla p \right\|_{L^2(\varOmega)}.
\end{displaymath}
\end{lemma}

\begin{proof}
We use a cut-off function as in the previous proof and denote $\eta_\ell\ceq\eta_j^{k,\ell}$ with $\ell\leq k-\varrho-\kappa-2$.

Applying Lemma~\ref{HMP::lma:cutoff} yields the existence of $\tilde{p}\in \Vf(\varOmega \setminus \omega_j^{k-\ell-\kappa-2})$ with the estimate $\|\nabla(\eta_\ell p -\tilde{p})\|_{L^2(\varOmega)}\lesssim \ell^{-1} \left\|\nabla p\right\|_{L^2(\omega_j^{k+2}\setminus\omega_j^{k-\ell-2})}$. Due to the property $\tilde{p}\in  \Vf(\varOmega \setminus \omega_j^{k-\ell-\kappa-2})$ and the assumptions on $F$ we also have
\begin{equation}
\label{HMP::lemm-4-4-step-1} \int_{\varOmega \setminus \omega_j^{k-\ell-\kappa-2}} A \nabla p \cdot \nabla \tilde{p}\de x =  \int_\varOmega A \nabla p \cdot \nabla \tilde{p}\de x = F ( \tilde{p} ) = 0.
\end{equation}
This leads to
\begin{align*}
\alpha\left\|\nabla p\right\|^2_{L^2(\varOmega\setminus\omega_j^k)}
&\leq\int_{\varOmega\setminus \omega_j^k}A\nabla p \cdot \nabla p\de x
\leq \int_{\varOmega\setminus \omega_j^{k-\ell-\kappa-2}}\eta_\ell A\nabla p \cdot \nabla p\de x\\
&= \int_{\varOmega\setminus \omega_j^{k-\ell-\kappa-2}} A\nabla p\cdot \left(\nabla (\eta_\ell p)-p\nabla\eta_\ell\right)\de x.
\intertext{With \eqref{HMP::lemm-4-4-step-1} and since $p\in\Vf$, this is}
&= \int_{\varOmega\setminus \omega_j^{k-\ell-\kappa-2}} A\nabla p\cdot \bigl(\nabla (\eta_\ell p-\tilde{p})-(p-I(p))\nabla\eta_\ell\bigr)\de x\\
&\lesssim \ell^{-1} \beta \Bigl(\|\nabla p \|_{L^2(\varOmega \setminus \omega_j^{k-\ell-\kappa-2})}^2 \\&\hspace{1cm}+ H^{-1} \|\nabla p \|_{L^2(\varOmega \setminus \omega_j^{k-\ell-\kappa-2})} \|p - I(p)\|_{L^2(\varOmega \setminus \omega_j^{k-\ell-\kappa-2})} \Bigr)\\
&\Stackrel{\textbf{(I3)}}\lesssim \ell^{-1} \beta \left\|\nabla p \right\|_{L^2(\varOmega \setminus \omega_j^{k-\ell-\kappa-2})}^2.
\end{align*}
Hence, there exists a constant $C$ independent of mesh size, contrast, number of patch extension layers, such that
\begin{equation}\label{HMP::eq:decay}
\|\nabla p \|_{L^2(\varOmega\setminus \omega_j^k)}^2\leq C \ell^{-1} \frac\beta\alpha \left\|\nabla p \right\|_{L^2(\varOmega\setminus \omega_j^{k-\ell-\kappa-2})}^2.
\end{equation}
Choose $\ell\ceq\lceil eC\frac\beta\alpha\rceil$ and observe that successive use of \eqref{HMP::eq:decay} yields
\begin{align*}
\|\nabla p \|_{L^2(\varOmega\setminus \omega_j^k)}^2
&\leq e^{-1} \left\|\nabla p \right\|_{L^2(\varOmega\setminus \omega_j^{k-\ell-\kappa-2})}^2\\
&\leq e^{-\lfloor\frac {k-\varrho}{\ell+\kappa+2}\rfloor} \left\|\nabla p \right\|_{L^2(\varOmega\setminus \omega_j^\varrho)}^2\\
&\lesssim e^{-\frac k{\ell+\kappa+2}} \left\|\nabla p \right\|_{L^2(\varOmega)}^2.
\end{align*}
The choice $\theta\ceq e^{-(\lceil eC\beta/\alpha\rceil +\kappa+2)^{-1}}$ concludes the proof.
\qedsquare\end{proof}

\begin{lemma}[localization error]
\label{HMP::lma:localization-error}
For $\uc\in\Vc$, the correction operators $Q^m$ and $Q^\varOmega$ satisfy
\begin{displaymath}
\bigl\| \nabla\bigl(Q^\varOmega\uc-Q^m\uc\bigr)\bigr\|_{L^2(\varOmega)} \lesssim \tfrac\beta\alpha m^{d/2}\tilde\theta^m \left\|Q^\varOmega \ucm \right\|_{L^2(\varOmega)}
\end{displaymath}
with $\tilde\theta\ceq\theta^{\lceil \hat H/H\rceil}<1$ and $\theta$ from Lemma~\ref{HMP::lma:decay}.
\end{lemma}

\begin{proof}
Recall the definition $Q^m\uc\ceq\sum_{\smash{\hj\in\hJ}} Q_\hj^m(\uc)$ with
\begin{displaymath}
\int_{\homega_\hj^m} A \nabla Q_\hj^m(\uc)\cdot \nabla w\de x = \underbrace{- \int_\varOmega {\hphi}_\hj A \nabla \uc \cdot \nabla w \de x}_{F_\hj(w)}\enspace \text{for all } w \in \Vf(\homega_\hj^m),\enspace\hj\in\hJ.
\end{displaymath}
Note that the right-hand side $F_\hj$ of the local problem is zero for $w\in\Vf(\varOmega\setminus\homega_\hj)$.
Consider some fixed $\hj\in\hJ$ and choose $j\in\J$ such that $\omega_j\cap\homega_\hj\neq\emptyset$. Recall $\varrho = \bigl\lceil\frac{\hat H}H\bigr\rceil$, then we have $\homega_\hj\subseteq\omega_j^\varrho$  and thus $\Vf(\varOmega\setminus\omega_j^\varrho)\subseteq\Vf(\varOmega\setminus\homega_\hj)$. Hence $F_\hj$ satisfies the conditions from Lemma~\ref{HMP::lma:decay}.

Moreover, we get
$\omega_j^k\subseteq\homega_\hj^m$ for $k$ satisfying
\begin{equation}\label{HMP::eq:m-lesssim-k}
 m = \bigl\lceil\tfrac{k\cdot H}{\hat H}\bigr\rceil \leq k\bigl\lceil\tfrac H{\hat H}\bigr\rceil .
\end{equation}

Denote $v\ceq Q^\varOmega\uc-Q^m\uc \in \Vf$ and note that $I(v)=0$. Using the cut-off functions $\eta_j^{k,1}$ from Definition~\ref{HMP::def:cutoff}, we obtain
\begin{multline*}
\alpha\bigl\| \nabla v \bigr\|_{L^2(\varOmega)}^2
\leq \sum_{\smash{\hj\in\hJ}} \Bigl(\,\underbrace{\bigl(A\nabla\bigl(Q^\varOmega_\hj\uc-Q^m_\hj\uc\bigr), \nabla (v (1- \eta_j^{k,1}))\bigr)_{L^2(\varOmega)}}_{\mathrm I} \\+ \underbrace{(A\nabla\bigl(Q^\varOmega_\hj\uc-Q^m_\hj\uc\bigr), \nabla(v \eta_j^{k,1}))_{L^2(\varOmega)}}_{\mathrm{II}}\,\Bigr).
\end{multline*}
We bound the term $\mathrm I$ by
\begin{align*}
\mathrm I &\le\beta\bigl\| \nabla\bigl(Q^\varOmega_\hj\uc-Q^m_\hj\uc\bigr)\bigr\|_{L^2(\varOmega)} \bigl\| \nabla\bigl( v (1- \eta_j^{k,1}) \bigr) \bigr\|_{L^2(\omega_j^{k})} \\
&\le\beta\bigl\| \nabla\bigl(Q^\varOmega_\hj\uc-Q^m_\hj\uc\bigr)\bigr\|_{L^2(\varOmega)} \bigl( \left\| \nabla v\right\|_{L^2(\omega_j^{k})} +\bigl\| v \nabla\bigl( 1- \eta_j^{k,1} \bigr)\bigr\|_{L^2(\omega_j^{k} \setminus \omega_j^{k-1} )} \bigr)\\
&\lesssim\beta\bigl\| \nabla\bigl(Q^\varOmega_\hj\uc-Q^m_\hj\uc\bigr)\bigr\|_{L^2(\varOmega)} \bigl(\left\| \nabla v\right\|_{L^2(\omega_j^{k})} + H^{-1}\left\| v - I(v)\right\|_{L^2(\omega_j^{k} \setminus \omega_j^{k-1} )} \bigr)\\
&\lesssim \beta \left\| \nabla\bigl(Q^\varOmega_\hj\uc-Q^m_\hj\uc\bigr) \right\|_{L^2(\varOmega)} \left\| \nabla v \right\|_{L^2( \omega_j^{k+1} )} .
\end{align*}
Lemma \ref{HMP::lma:cutoff} yields the existence of $\tilde{v}\in \Vf(\varOmega \setminus \omega_j^{k-\kappa-3})$ with $$\bigl\| \nabla (v \eta_j^{k,1} - \tilde{v}) \bigr\|_{L^2(\varOmega)} \lesssim \left\| \nabla v \right\|_{L^2(  \omega_j^{k+2} )}.$$ We assume that $m$ is large enough such that $k\geq \varrho+\kappa+3$, then $\tilde v\in\Vf(\varOmega\setminus\homega_\hj)$ and hence
$$\int_\varOmega A \nabla\bigl(Q^\varOmega_\hj\uc-Q^m_\hj\uc\bigr) \cdot \nabla \tilde{v}\de x=0.$$
It follows that
\begin{align*}
\mathrm{II} &= \bigl(A\nabla\bigl(Q^\varOmega_\hj\uc-Q^m_\hj\uc\bigr), \nabla(v \eta_j^{k,1}-\tilde{v})\bigr)_{L^2(\varOmega)}\\
&\lesssim \beta\left\|\nabla\bigl(Q^\varOmega_\hj\uc-Q^m_\hj\uc\bigr) \right\|_{L^2(\varOmega)} \left\| \nabla v \right\|_{L^2( \omega_j^{k+2} )}.
\end{align*}
Combining the estimates for $\mathrm I$ and $\mathrm{II}$ finally yields
\begin{align}
\label{HMP::lemma-a-3-proof-eq-0}\bigl\| \nabla v \bigr\|_{L^2(\varOmega)}^2
&\lesssim \frac\beta\alpha\sum_{\smash{\hj\in\hJ}} \bigl\| \nabla\bigl(Q^\varOmega_\hj\uc-Q^m_\hj\uc\bigr) \bigr\|_{L^2(\varOmega)} \left\| \nabla v \right\|_{L^2( \omega_j^{k+2} )}\\
\notag&\lesssim \frac\beta\alpha \,k^{d/2} \Bigl( \sum_{\smash{\hj\in\hJ}} \bigl\| \nabla\bigl(Q^\varOmega_\hj\uc-Q^m_\hj\uc\bigr) \bigr\|_{L^2(\varOmega)}^2 \Bigr)^{1/2} \left\| \nabla v \right\|_{L^2(\varOmega)},
\end{align}
provided that $\left|\{i\in\J\mid\omega_i\subseteq\omega_j^{k+2}\}\right|\lesssim k^{d/2}$.

In order to bound $\bigl\| \nabla\bigl(Q^\varOmega_\hj\uc-Q^m_\hj\uc\bigr) \bigr\|_{L^2(\varOmega)}^2$, we use Galerkin orthogonality for the  local problems, which is
\begin{equation}
\label{HMP::galerkin-orthogonality-local-eq}\bigl\| \nabla\bigl(Q^\varOmega_\hj\uc-Q^m_\hj\uc\bigr) \bigr\|_{L^2(\varOmega)}^2 \lesssim \inf_{q \in \Vf(\omega_j^k) } \bigl\| \nabla (Q^\varOmega_\hj\uc-q ) \bigr\|_{L^2(\varOmega)}^2.
\end{equation}
\textbf{(I4)} yields the existence of $\tilde w\in \Vf$ such that 
\begin{gather*}
I (\tilde w)=I ((1-\eta_j^{k,1}) Q^\varOmega_\hj\uc ),\quad\|\nabla \tilde w\|_{L^2(\varOmega)}\lesssim \|\nabla I ((1-\eta_j^{k,1})Q^\varOmega_\hj\uc)\|_{L^2(\varOmega)},\\
\text{and}\quad\notag\supp(\tilde w)\subsetsim[\kappa] \supp((1-\eta_j^{k,1})Q^\varOmega_\hj\uc) \subseteq \omega_j^k.
\end{gather*}
We observe
\begin{equation}
\label{HMP::lemma-a-3-proof-eq-1}
\bigl\|\nabla I ((1-\eta_j^{k,1})Q^\varOmega_\hj\uc)\bigr\|_{L^2( \omega_j^{k+\kappa} )}^2
= \bigl\|\nabla I ((1-\eta_j^{k,1})Q^\varOmega_\hj\uc)\bigr\|_{L^2( \omega_j^{k+1} \setminus  \omega_j^{k-2} )}^2.
\end{equation}
With $p_\hj\ceq(1-\eta_j^{k,1})Q^\varOmega_\hj\uc-\tilde w\in \Vf(\omega_j^{k+\kappa})$, we obtain
\begin{align}
\notag \bigl\| \nabla\bigl(Q^\varOmega_\hj\uc-Q^m_\hj\uc\bigr) \bigr\|_{L^2(\varOmega)}^2
&\Stackrel{\eqref{HMP::galerkin-orthogonality-local-eq}}\lesssim \bigl\| \nabla (\eta_j^{k,1} Q^\varOmega_\hj\uc+ (1-\eta_j^{k,1})Q^\varOmega_\hj\uc - p_\hj)  \bigr\|_{L^2(\varOmega)}^2 \\
\notag&= \bigl\| \nabla (\eta_j^{k,1} Q^\varOmega_\hj\uc - \tilde w)  \bigr\|_{L^2(\varOmega)}^2
\\
\notag &\lesssim \|\nabla Q^\varOmega_\hj\uc\|_{L^2(\varOmega \setminus \omega_j^{k-2} )}^2 +  \|\nabla \tilde w\|_{L^2( \omega_j^{k+\kappa} )}^2\\
\notag &\lesssim\|\nabla Q^\varOmega_\hj\uc\|_{L^2(\varOmega \setminus \omega_j^{k-2} )}^2 \\\notag&\hspace{5em}
+ \|\nabla I ((1-\eta_j^{k,1})Q^\varOmega_\hj\uc)\|_{L^2( \omega_j^{k+\kappa} )}^2\\
\notag &\Stackrel{\eqref{HMP::lemma-a-3-proof-eq-1}}\lesssim \|\nabla Q^\varOmega_\hj\uc\|_{L^2(\varOmega \setminus \omega_j^{k-2} )}^2 \\\notag&\hspace{5em}
+ \|\nabla I ((1-\eta_j^{k,1})Q^\varOmega_\hj\uc)\|_{L^2( \omega_j^{k+1} \setminus  \omega_j^{k-2} )}^2
\\ 
\notag &\Stackrel{\textbf{(I3)}}\lesssim \|\nabla Q^\varOmega_\hj\uc\|_{L^2(\varOmega \setminus \omega_j^{k-3} )}^2 \\
\notag &\Stackrel{\text{Lemma~\ref{HMP::lma:decay}}}\lesssim \hspace{1em} \theta^{2 (k-3)} \bigl\|\nabla Q^\varOmega_\hj\uc \bigr\|_{L^2(\varOmega)}^2\\
&\Stackrel{\eqref{HMP::eq:m-lesssim-k}}\lesssim \ \tilde\theta^{2m} \bigl\|\nabla Q^\varOmega_\hj\uc \bigr\|_{L^2(\varOmega)}^2.
\label{HMP::lemma-a-3-proof-eq-3} 
\end{align}
Combining \eqref{HMP::lemma-a-3-proof-eq-0} and \eqref{HMP::lemma-a-3-proof-eq-3} proves the lemma.
\qedsquare\end{proof}

\section{Numerical Experiment}
\label{s:numericalexperiments}

In this section, we present numerical results for a special realization of the Multiscale Partition of Unity Method. We consider a ``coarse'' regular triangulation $\mathcal T_H$ of $\varOmega$, where $H$ denotes the maximum diameter of an element of $\mathcal T_H$. By $\mathcal{N}_H$ we denote the set of vertices of the triangulation. 
We choose the basis functions $\varphi_z$ as in Example~\ref{HMP::ex:1}b), i.e., the continuous and piecewise affine nodal basis functions associated with vertices $z\in\mathcal N=\J$. The second partition of unity (PU 2) is given by the indicator functions of the elements of the triangulation, i.e. $\{ \hphi_\hj\mid\hj \in \hJ \}\ceq\{\chi_T\mid\hspace{2pt} T \in \mathcal T_H\}$. The corrector problems given by \eqref{HMP::local-corrector-problem} are solved with a $P_1$ Finite Element method on a fine grid with resolution $h=2^{-8}$. The reference solution $u_h$ is therefore the $P_1$ Finite Element approximation in a space with mesh size $h=2^{-8}$. 

In order to estimate the accuracy of $u_h$ itself, we performed a second computation for the mesh size $h=2^{-10}$.
The relative $L^2$-error between the Finite Element approximation on a uniform mesh with resolution $h=2^{-8}$ and the Finite Element approximation on a uniform mesh with resolution $h=2^{-10}$ is $0.023$. The relative $H^1$-error is $0.3204$. However, we only compute the errors of $\ucm$ with respect to the reference solution (i.e. for $h=2^{-8}$), since this is the relevant error for investigating the effect of the coarse grid resolution and the decay of the multiscale basis functions on $\ucm$.

The extension patches $\homega_\hj^m$ can be defined by using the structure of the coarse grid by setting
\begin{equation}\label{truncation-fem}
\begin{split}
\homega_\hj^0 & \ceq T_j \in \mathcal T_H, \\
\homega_\hj^m & \ceq \cup\{T\in \mathcal T_H\;\vert\; T\cap \homega_\hj^{m-1}\neq\emptyset\}\quad m=1,2,\ldots .
\end{split}
\end{equation}
\begin{figure}[!ht]
\centering
\includegraphics[width=0.7\textwidth]{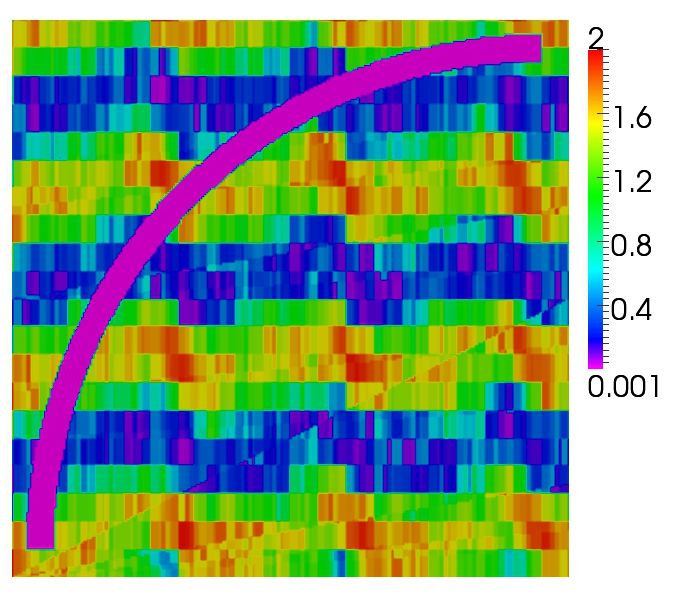}
\caption{ Plot of the rapidly varying and highly heterogeneous diffusion coefficient $a_\varepsilon$ given by equation \eqref{diffusion-coefficient}, which takes values between $0.01$ and $2$. The structure is disturbed by an isolating arc (purple) of thickness $0.05$ and with conductivity $10^{-3}$.}
\label{diffusion_problem_1}
\end{figure}

We consider the following model problem. Let $\varOmega \ceq \left]0,1\right[^2$ and $\varepsilon\ceq0.05$. Find $u_\varepsilon\in V$ with
\begin{alignat}{2}
\label{model-problem-numerics}- \Div \bigl( a_\varepsilon(x) \nabla u_\varepsilon(x) \bigr) 
&= x_1 - \tfrac12 &\qquad& \text{in } \varOmega \\
\nonumber\nabla u_\varepsilon(x) \cdot \nu &= 0 && \text{on } \partial \varOmega.
\end{alignat}
The scalar diffusion coefficient $a_\varepsilon$ in equation \eqref{model-problem-numerics} is depicted in Figure \ref{diffusion_problem_1}. It has a contrast of order $10^3$ and is constructed from the highly heterogeneous distribution 
\begin{displaymath}
c_\varepsilon(x_1,x_2)\ceq1 + \tfrac1{10} \sum_{j=0}^4 \sum_{i=0}^{j} \left( \tfrac{2}{j+1} \cos \left(\bigl\lfloor i x_2 - \tfrac{x_1}{1+i} \bigr\rfloor + \left\lfloor \tfrac{i x_1}\varepsilon \right\rfloor + \left\lfloor \tfrac{ x_2}\varepsilon \right\rfloor \right)\right)
\end{displaymath}
and an isolating arc of radius $r\ceq0.9$, thickness $\frac\varepsilon2$ and center $c_0\ceq(1-\varepsilon,\varepsilon)$. The coefficient $a_\varepsilon$ is then given by
\begin{gather}
\label{diffusion-coefficient}a_\varepsilon(x)\ceq \begin{cases}
10^{-3} &\text{if} \enspace \bigl||x-c_0| - r\bigr|<\frac\varepsilon{2}, \enspace x_2>\varepsilon \enspace\text{and} \enspace x_1 < 1-\varepsilon \\
(h \circ c_\varepsilon)(x) &\text{else,}\\
\end{cases}\\
\nonumber \text{with} \enspace h(t)\ceq\begin{cases}
t^4 &\text{for} \enspace \frac12 < t < 1 \\ 
t^{\frac{3}{2}} &\text{for} \enspace 1 < t < \frac{3}{2}  \\ 
t &\text{else}.
\end{cases}
\end{gather}

\begin{figure}[!ht]
\centering
\includegraphics[width=1.0\textwidth]{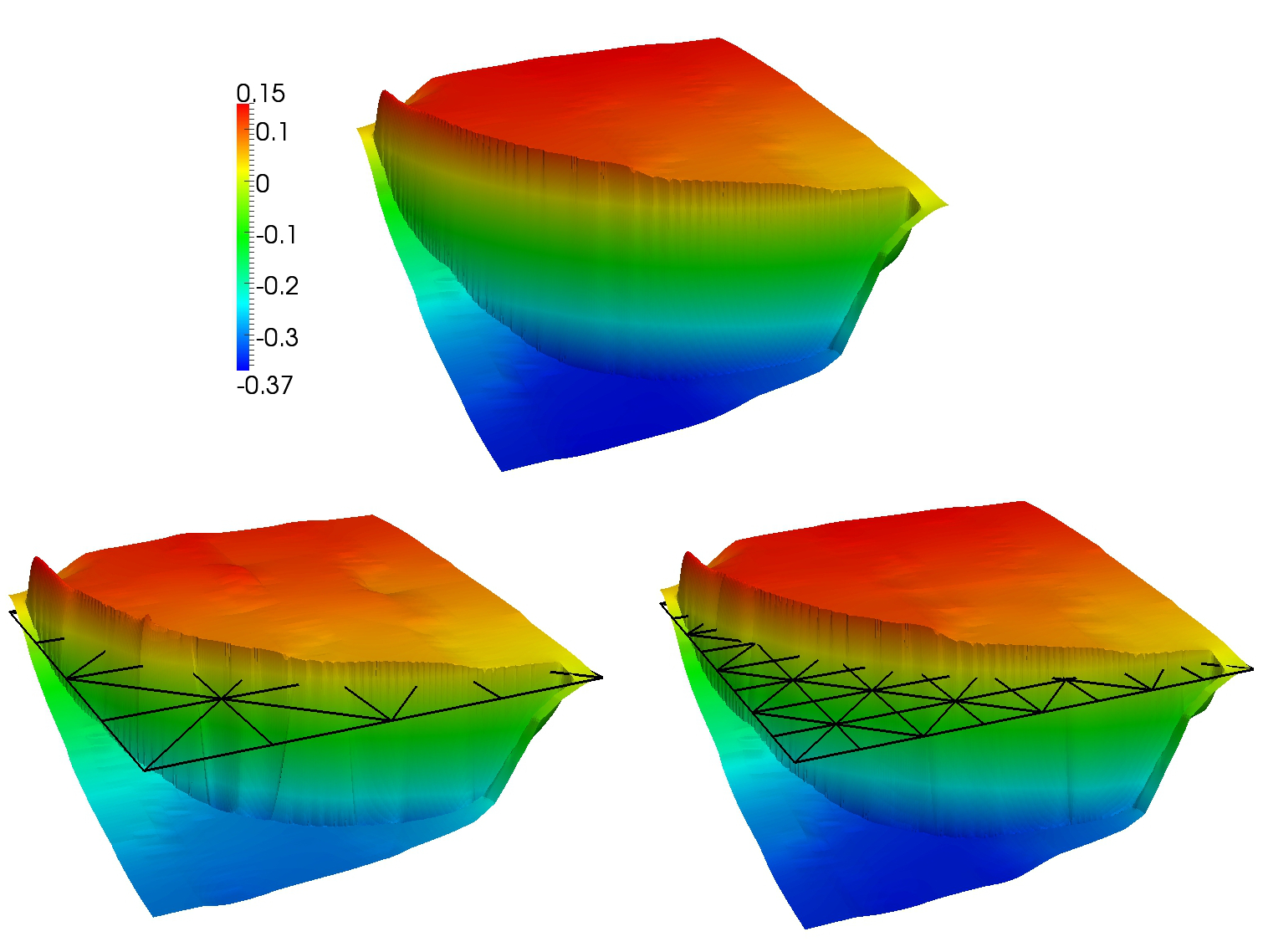}
\caption{ The top picture shows the $P_1$ finite element reference solution $u_h$ for $h=2^{-8}$. The left bottom picture shows the multiscale approximation $\ucm$ for $(H,m)=(2^{-2},1)$ together with the corresponding coarse grid. This solution already shows the essential features of $u_h$. The right bottom picture shows the multiscale approximation $\ucm$ for $(H,m)=(2^{-3},2)$ together with the corresponding coarse grid. }
\label{series-lod-warp}
\end{figure}

In our computation, we picked the truncation parameter $m$ (according to \eqref{truncation-fem}) to be in the span between $0$ and $2$ and the coarse mesh size $H$ to be in the span between $2^{-1}$ (i.e. $h=H^8$) and $2^{-4}$ (i.e. $h=H^2$). The results are depicted in Table~\ref{table-layers-results}. We observe that error stagnates if we decrease only $H$, without increasing $m$ at the same time. However, already the modification $(H,m)=(2^{-m-1},m) \mapsto (2^{-m-2},m+1)$ leads to a dramatic error reduction. Despite the high contrast of order $10^3$, we already obtain a highly accurate approximation for $(H,m)=(2^{-3},2)$. In this case, the multiscale approximation looks almost identical to the FEM reference solution for $h=2^{-8}$ (see Figure \ref{series-lod-warp}). Further numerical experiments can be found in \cite{HMP12,HP13,MP11}.
 
\begin{table}[t]
\caption{Results for the relative error between the Multiscale Partition of Unity approximation $\ucm$ and a reference solution $u_h$ on a fine grid of mesh size $h=2^{-8}\approx 0.0039 \ll \varepsilon$ which fully resolves the micro structure of the coefficient $a_\varepsilon$. We use the notation \mbox{$\|\ucm - u_h\|_{L^2(\varOmega)}^{\mbox{rel}}:=\|\ucm - u_h\|_{L^2(\varOmega)}/\|u_h\|_{L^2(\varOmega)}$} and analogously the same for $\|\ucm - u_h\|_{H^1(\varOmega)}^{\mbox{rel}}$. The truncation parameter $m$ determines the patch size and is given by (\ref{truncation-fem}).}
\label{table-layers-results}
\begin{center}
\begin{tabular}{|c|c|c|c|c|}
\hline $H$      & $m$ & $\|\ucm - u_h\|_{L^2(\varOmega)}^{\mbox{rel}}$ & $\|\ucm - u_h\|_{H^1(\varOmega)}^{\mbox{rel}}$ \\
\hline
\hline $2^{-1}$ & 0 & 0.867827  & 0.93475 \\
\hline $2^{-2}$ & 0 & 0.865630  & 0.96525 \\
\hline $2^{-2}$ & 1 & 0.167501  & 0.37387  \\
\hline $2^{-3}$ & 1 & 0.257826  & 0.61681  \\
\hline $2^{-3}$ & 2 & 0.037841  & 0.16525  \\ 
\hline $2^{-4}$ & 2 & 0.063645  & 0.25613  \\
\hline
\end{tabular}\end{center}
\end{table}

\bibliographystyle{amsplain}
\bibliography{HMP_references.bib}

\end{document}